\documentclass[reqno,12pt]{amsart}
\usepackage{amsmath,amsthm,amssymb,amsfonts,amscd}
\usepackage{epsfig}
\usepackage{color}
\usepackage[all]{xy}
\usepackage{tikz}
\usepackage{hyperref}
\usepackage{amsmath}
\usepackage{mathabx}
\allowdisplaybreaks

\theoremstyle{plain}
\newtheorem{theorem}{Theorem}

\newtheorem{corollary}[theorem]{Corollary}
\newtheorem{proposition}[theorem]{Proposition}
\newtheorem*{theorem*}{Theorem}
\newtheorem*{conjecture*}{Conjecture}

\theoremstyle{definition}
\newtheorem{remark}[theorem]{Remark}

\usepackage{graphics, setspace}

\newcommand{\breakingcomma}{%
  \begingroup\lccode`~=`,
  \lowercase{\endgroup\expandafter\def\expandafter~\expandafter{~\penalty0 }}}
  
\usepackage{breqn}

\usepackage{amstext} 
\usepackage{array}
\newcolumntype{L}{>{$}l<{$}}

\newcommand{\CC}{{\mathbb{C}}}

\newcommand{\QQ}{{\mathbb{Q}}}

\newcommand{\ZZ}{{\mathbb{Z}}}

\newcommand{\SL}{{\mathrm{SL}}}
\newcommand{\J}{{{\langle J \rangle}}}

\newcommand{\Jac}{\mathrm{Jac}}
\newcommand{\Fix}{\mathrm{Fix}}

\newcommand{\id}{\mathrm{id}}

\newcommand{\bx}{{\bf x}}

\newcommand{\age}{\mathrm{age}}

\def\A{{\mathcal A}}

\def\C{{\mathcal C}}

\DeclareMathOperator{\sgn}{sgn}

\newcommand{\ccHH}{{\mathsf{HH}}}

\newcommand\sig{\sigma}

\begin{document}
\title{Mirror map for Fermat polynomial with non--abelian group of symmetries}
\date{\today}
\author{Alexey Basalaev}
\address{A. Basalaev:\newline Faculty of Mathematics, National Research University Higher School of Economics, Usacheva str., 6, 119048 Moscow, Russian Federation, and \newline
Skolkovo Institute of Science and Technology, Nobelya str., 3, 121205 Moscow, Russian Federation}
\email{a.basalaev@skoltech.ru}
\author{Andrei Ionov}
\address{A. Ionov:\newline Faculty of Mathematics, National Research University Higher School of Economics, Usacheva str., 6, 119048 Moscow, Russian Federation, and \newline
Department of Mathematics, Massachusetts Institute of Technology, 77 Massachusetts
Ave., Cambridge, MA 02139, United States}
\email{aionov@mit.edu}

\begin{abstract}

We study Landau-Ginzburg orbifolds $(f,G)$ with $f=x_1^n+\ldots+x_N^n$ and $G=S\ltimes G^d$, where $S\subseteq S_N$ and $G^d$ is either the maximal group of scalar symmetries of $f$ or the intersection of the maximal diagonal symmetries of $f$ with $\SL_N(\CC)$. We construct a mirror map between the corresponding phase spaces and prove that it is an isomorphism restricted to a certain subspace of the phase space when $n=N$ is a prime number. When $S$ satisfies the condition PC of Ebeling-Gusein-Zade this subspace coincides with the full space. 
We also show that two phase spaces are isomorphic for $n=N=5$.

\end{abstract}
\maketitle


\section{Introduction}
Initiated in the late 80s by the physicists (cf. \cite{IV90,V89,W93}), lots of effort has been given to the study of the so--called \textit{Landau--Ginzburg orbifolds}. These are the pairs $(f,G)$, for $f = f(\bx) \in \CC[x_1,\dots,x_N]$ having only isolated critical points and $G$ being a \textit{group of symmetries} of $f(\bx)$. Namely, a group of elements $g \in \mathrm{End}(\CC^N)$, s.t. $f(g \cdot \bx) = f(\bx)$. These pairs appeared to play an imporant role in mirror symmetry. One associates two vector spaces to every such pair: an \textit{A--vector space} and \textit{B--vector space}, both being of the same dimension, but different in the construction. 

Mirror symmetry conjectures that there is a \textit{dual} pair $(\widetilde f, \widetilde G)$, such that its A--vector space (resp. B--vector space) is isomorphic to the B--vector space (resp. A--vector space) of the pair $(f,G)$. Such an isomorphism is called \textit{mirror map}. All the other mirror symmetry results can only be obtained after the good mirror map is settled.

In this paper we focus on the Fermat polynomial case
\[
 f = x_1^n + \dots + x_N^n.
\]
For it we have $\widetilde f(\bx) = f(\bx)$, however the construction of the dual group $\widetilde G$ is more involved.

The B--vector space is given by the Hochschild cohomology $\ccHH^\ast(f,G)$ of the category of $G$--equivariant matrix factorizations of $f$. Denote by $f^g$ the restriction of $f$ to $\Fix(g)$ --- eigenvalue $1$ subspace of $g$ in $\CC^N$, and set $\A_{f,g}' := \Jac(f^g)$ --- the Jacobian algebra of $f^g$. We have (cf. \cite{S20})
\[
    \ccHH^\ast(f,G) \cong \left( \A'_{f,G} \right)^G \ \text{ for} \ \A'_{f,G} := \bigoplus_{g \in G} \A'_{f,g}.
\] 
The action of $v \in G$ is s.t. we have $v^*: \A'_{f,u} \to \A'_{f,v u v^{-1}}$ for any $u \in G$.
The subspaces $\A'_{f,u}$ are called \textit{narrow sectors} if $\Fix(u) = 0$ and \textit{broad sectors} otherwise.

The A--vector space is constructed via the FJRW theory that is only defined for the groups $G$ acting diagonally (cf. \cite{FJR}, see also \cite{WWP} for some examples). However, one expects it to have the same vector space structure as $\ccHH^\ast(f,G)$ above. This allows one to consider mirror map as an involutive vector space isomorphism $\ccHH^\ast(f,G) \cong \ccHH^\ast(\widetilde f, \widetilde G)$.

\subsection{Diagonal symmetry groups}
Consider the \textit{maximal group of diagonal symmetries} 
\[  
G_f^d := \lbrace g = (g_1,\dots,g_N) \in (\CC^\ast)^N \ \mid \ f(g \cdot \bx) = f(\bx) \rbrace.
\]
We call $G$ a diagonal symmetry group if $G \subseteq G_f^d$. 
For such groups 
the notion of a dual group $\widetilde G$ was introduced by Berglund-H\"ubsch-Henningson in \cite{BH95,BH93} and the mirror map was found by M.Krawitz in \cite{K09}. The mirror map of Krawitz always interchanges broad and narrow sectors.

Especially important are the groups $\SL_f$ and $\J$, for
\begin{align}
    \SL_f &:= \lbrace g = (g_1,\dots,g_N) \in G_f^d \ \mid \ \prod_{i=1}^N g_i = 1 \rbrace,
    \\
    J &:= (\exp(2 \pi \sqrt{-1}/ n ), \dots, \exp(2 \pi \sqrt{-1}/n )).
\end{align}
These groups are dual to each other.

Diagonal symmetry group $G$ is always abelian what simplifies significantly the computations. In particular, the $G$--action does not mix up the sectors of $\A'_{f,G}$. 
For $f$ being invertible polynomial the corresponding FJRW and Hochschild cohomology rings were computed in \cite{FJJS,BT2,BTW16,BTW17}.

\subsection{Fermat quintic with nonabelian group of symmetries}
In \cite{IV90} the authors introduced two functions $q_l,q_r: \A'_{f,G} \to \QQ$ that provide the bigrading of $\ccHH^*(f,G)$. This bigrading was used in \cite{Muk} to get the Hodge numbers $h^{p,q}(f,G)$ for a B--vector space. For $N=n=5$ and $G$ preserving the volume form, Mukai has shown (Theorem 7.1 in loc.cit.) that 
$h^{p,q}(f,G) = h^{3-p,q}(X)$, for $X$ being the mirror Calabi-Yau quintic of $(f,G)$.

This approach was extended further in \cite{EGZ18,EGZ20} by Ebeling and Gusein--Zade who have considered the mirror pairs $(f, S \ltimes G^d)$, $(\widetilde f, S \ltimes \widetilde G^d)$ with $\J \subseteq G^d \subseteq \SL_f$ from the point of view of Hodge theory of the Milnor fibres. 
They have shown that for two mirror Calabi-Yau quintics $X$ and $\widecheck{X}$ of these pairs, the equality $h^{1,1}(X) = h^{2,1}(\widecheck{X})$ holds if and only if the following \textit{parity condition holds}.
\begin{equation}\label{eq: PC}
     \text{for any } T \subseteq S \text{ holds } \dim \left( \CC^N \right)^T \equiv N \mod 2.
     \tag{PC}
\end{equation}
In particular, $S$ can only satisfy PC if $S \subseteq A_N$.
This condition was conjectured to be necessary for mirror symmetry to hold.


\subsection{In this paper}
We focus in this paper on the case of prime $N$ and  nonabelian symmetry groups $G = S \ltimes \SL_f$ and $\widetilde G = S \ltimes \langle J \rangle$ with $S \subseteq S_N$. The corresponding Hochschild cohomology groups were computed in \cite{BI} via the technique of \cite{S20}.

We introduce the mirror map $\tau: \ccHH^\ast(f, S \ltimes \SL_f) \to \ccHH^\ast(f, S \ltimes \langle J \rangle)$. Our mirror map coincides with the mirror map of Krawitz on $\A'_{f,u}$, s.t. $u \in G^d$.

It is easy to see that there is no mirror map for $n \neq N$ because the dimensions do not match (see Example~1 in Section~\ref{section: mirror map examples}). 
For $n=N$ we show (Theorem~\ref{theorem: general} in the text).
\begin{theorem}\label{theorem 1}
    The map $\tau$ establishes an isomorphism $\A_{f,S \ltimes \SL_f}^{stable} \to \A_{f,S \ltimes \J}^{stable}$ between the certain subspaces of $\ccHH^\ast(f, S \ltimes \SL_f)$ and $\ccHH^\ast(f, S \ltimes \langle J \rangle)$ respectively. 
    In particular, these subspaces coincide with the whole vector spaces if the group satisfies \eqref{eq: PC} of Ebeling--Gusein-Zade.
    
    Under this isomorphism we have $q_l(X) = q_l(\tau(X))$ and $q_r(X) = N-2-q_r(\tau(X))$
    for any homogeneous $X \in \A_{f,S \ltimes \SL_f}^{stable}$.
\end{theorem}

It's important to note that our mirror map establishes the isomorphism between the whole spaces $\ccHH^\ast(f, S \ltimes \SL_f)$ and $\ccHH^\ast(f, S \ltimes \langle J \rangle)$ for some examples, for which PC does not hold.

For the special case $N=n=5$ we can make even a stronger result (see Theorem~\ref{theorem: quintic} and it's corollary in the text).

\begin{theorem}\label{theorem 2}
    For $N=5$ there is a vector space isomorphism ${\ccHH^\ast(f, S \ltimes \SL_f) \to \ccHH^\ast(f, S \ltimes \langle J \rangle)}$ for any group $S \subseteq S_5$. 
    
    Moreover, for $X$ and $\widecheck{X}$ being the mirror quintics of $(f,S\ltimes\SL)$ and $(f,S\ltimes\J)$ respectively, we have the following relation between the Hodge numbers
    \[
        h^{(1,1)}(X) + h^{(2,1)}(X) = h^{(1,1)}(\widecheck{X}) + h^{(2,1)}(\widecheck{X}).
    \]
\end{theorem}
This theorem makes use of another mirror map, that generalizes the mirror map $\tau$ and also the mirror map of Krawitz, but does not always send the narrow sectors to broad and vise versa.

We also provide many examples in  Section~\ref{section: mirror map examples} and Section~\ref{section: examples2}.

\subsection*{Acknowledgement}~
\\
The authors acknowledge partial support by RSF grant no. 19-71-00086 and by International Laboratory of Cluster Geometry HSE University, RF Government grant. 
In particular, the proof of Theorem~\ref{theorem 1} was obtained under the support of RSF grant no. 19-71-00086 
and the proof of Theorem~\ref{theorem 2} was obtained under the support of International Laboratory of Cluster Geometry HSE University, RF Government grant.

The authors are grateful to anonymous referee for many important remarks.

\section{Preliminaries and notation}\label{section: preliminaries}
For any $u \in S_N \ltimes G_f^d$ we will denote $u = \sigma \cdot g$ assuming that $\sigma \in S_N$ and $g \in G_f^d$.
Let $\Fix(u)$ be the eigenvalue $1$ subspace of $\CC^N$ of $u$ and $I_u^c$ be the set of all indices $k$, s.t. $u \cdot x_k \neq x_k$. Restriction of $f$ to $\Fix(u)$, $f^u := f \mid_{\Fix(u)}$ is a Fermat type polynomial again.

Let $\sigma = \prod_{a=1}^p \sigma_a$ be the decomposition into the non--intersecting cycles. Denote by $|\sigma_a|$ the length of the cycle $\sigma_a$. We will also allow $\sigma_a$ to be of length $1$, so that we always have $\sum_{a=1}^p | \sigma_a | = N$. 
There exists the unique set $g_1,\dots,g_p$ of $G_f^d$--elements, s.t. $g_a$ acts non--trivially only on $I^c_{\sigma_a}$ and $\sigma \cdot g = \prod_{a=1}^p \sigma_a g_a$. We call the product $\sigma \cdot g = \prod_{a=1}^p \sigma_a g_a$ \textit{generalized cycle decomposition} of $u$.

A generalized cycle $\sigma_ag_a$ is said to be \textit{special} if $\det(g_a) = 1$ and \textit{non--special} otherwise. It is clear that $\Fix(\sigma_ag_a)\cap \CC^{I_{\sigma_ag_a}^c}= 0$ for a non--special cycle, where by $\CC^{I_{u}^c}$ we mean the subspace of $\CC^N$ spanned by standard basis vectors with indices in $I_{u}^c$. For a special cycle we have $\dim\Jac(f^{\sigma_ag_a}\mid_{\CC^{I_{\sigma_ag_a}^c}}) = n-1$. 
Denote by $\lfloor \phi(\bx) \rfloor$ the class of the polynomial $\phi(\bx)$ in $\Jac(f^{\sigma_ag_a})$.
Let $\widetilde x_{i_a}$ be the $\sigma_ag_a$-invariant linear combination of $x_\bullet$ with indexes in $I_{\sigma_ag_a}^c$, s.t. $\Jac(f^{\sigma_ag_a}|_{I_{{\sigma_ag_a}}^c})$ has the basis $\lfloor \widetilde x_\bullet^k\rfloor$, $k=0,\dots,n-2$. Set
\[
    \A'_{\sigma_ag_a} := \langle \lfloor 1 \rfloor, \lfloor \widetilde x_{i_a} \rfloor, \dots, \lfloor \widetilde x_{i_a}^{n-2} \rfloor \rangle \xi_{\sigma_ag_a},
\]
where we denote by $\xi_{\sigma_ag_a}$ the formal letter associated to $\sigma_ag_a$. 
The elements of $\A'_{\sigma_ag_a}$ will be denoted by $\lfloor \phi(\bx) \rfloor \xi_{\sigma_ag_a}$.

In particular, for $g_a = \id$ we have $\widetilde x_{i_a} = \sum_{i} x_i$ where the summation is taken over $i \in I^c_{\sigma_a}$. We adopt the notation above for the non--special cycles too, assuming $\lfloor \widetilde x_{i_a}^0 \rfloor \xi_{\sigma_a g_a} = \lfloor 1 \rfloor \xi_{\sigma_a g_a}$. For $u\in G$ with generalized cycle decomposition $u= \prod_{a=1}^p \sigma_a g_a$ we have 
\[
    \A'_{f,u}=\bigotimes_{a=1}^p \A'_{\sigma_ag_a}.
\]    

\indent Fix $\zeta_n := \exp ( 2 \pi \sqrt{-1})/n )$ and $t_k \in G_f^d$ with $k=1,\dots,N$ by 
\[
    t_k: (x_1,\dots,x_N) \to (x_1,\dots, \zeta_n x_k, \dots, x_N).
\]
Then the Fermat type polynomial maximal diagonal symmetries group $G_f^d$ is generated by $t_1,\dots,t_N$. Denote also
\[
    \SL_f := \left\lbrace g \in G_f^d \ | \ \det(g) = 1 \right\rbrace, \quad J := t_1\cdots t_N.
\]
The groups $S \ltimes \SL_f$ and $S \ltimes \J$ with $S \subseteq S_N$ will be in particular important in this paper.

\subsection{The phase space}
In what follows we need to consider $\ccHH^*(f,G)$ as a vector space for the fixed $S \subseteq S_N$ and different $G^d \subseteq G_f^d$. To make it easier we consider the space $\A_{tot} := \bigoplus_{u \in S \ltimes G_f^d} \A'_{f,u}$. We call its direct summand $\A'_{f,u}$ the \textit{$u$-th sectors}. 

For any $G \subseteq S_N \ltimes G_f^d$ we denote by $\A_{f,G}$ the \textit{phase space} of $G$, being the subspace of $\A_{tot}$ defined as follows.
Let $\C^G$ stand for the set of representatives of the conjugacy classes of $G$. 
Denote
\begin{equation}\label{eq: full algebra decomposition}
    \A_{f,G} :=  \bigoplus_{u \in \C^G} \left( \A_{f,u}' \right)^{Z(u)},
\end{equation}
where the action of $v \in Z(u)$ on $\A_{f,u}'$ is computed as follows.
Let $\lambda_k^u,\lambda_k^v$ be the eigenvalues of $u$ and $v$ computed in their common eigenvectors basis. For $X = \lfloor \phi(\bx) \rfloor \xi_u \in \A_{f,u}'$ and $v \in Z(u)$ we have
\begin{equation}\label{eq: group action}
    v^*\left( X \right) = \prod_{\substack{ k=1,\dots,N \\ \lambda_k^u \neq 1}} \frac{1}{\lambda_k^v} \lfloor \phi(v \cdot \bx) \rfloor \xi_u = ( \det(v_{\mid_{\Fix(u)}}) )^{-1} \lfloor \phi(v \cdot \bx) \rfloor \xi_u.
\end{equation}
This is the particular case of the $G$--action of $\ccHH^*(f,G)$. 
Moreover we have (cf. Proposition~42 in \cite{BI})
\begin{equation}
    \ccHH^*(f,G) \cong \A_{f,G}.
\end{equation}
This isomorphism allows us to consider $\ccHH^*(f,G)$ with the different groups $G$ as the subspaces of $\A_{tot}$.

For any $G_1,G_2 \subseteq S_N \ltimes G_f^d$ we have the natural inclusion $i_1: \A'_{f,G_1}\to\A_{tot}$ and the projections $\pi_2: \A_{tot} \to \A'_{f,G_2}$. In what follows we will consider the maps $\psi: \A'_{f,G_1} \to \A'_{f,G_2}$ via the maps $\widetilde \psi: \A_{tot} \to \A_{tot}$ by $\psi := \pi_2 \circ \widetilde \psi \circ i_1$.

With respect to a generalized cycle decomposition $u = \prod_a \sigma_ag_a$ we have the relation $\xi_u = \prod_{a=1}^p \xi_{\sigma_ag_a}$ between the generators of the different vector spaces $\A_{f,G_1}$, $\A_{f,G_2}$. This extends to the product of arbitrary $X_1 = \lfloor \phi_1 \rfloor \xi_{u}$ and $X_2 = \lfloor \phi_2 \rfloor \xi_{v}$ assumed as $\A_{tot}$--elements by $X_1X_2 := \lfloor \phi_1 \phi_2 \rfloor \xi_{uv}$ when $I_u^c \cap I_v^c = \emptyset$. This is not to be confused with the cup-product on Hochscild chohomology as in \cite{BI}.

\subsection{Bigrading}\label{section: bigrading}
Let $\CC[x_1,\dots,x_N]$ be graded by setting $\deg(x_k) := 1/n$. Extend this grading to $\A_{tot}$ as follows. For any $u \in S_N\ltimes G_f^d$ let $\lambda_1,\dots,\lambda_N \in \CC$ be the eigenvalues of the linear transformation $\bx \mapsto u \bx$. We may assume $\lambda_k = \exp(2 \pi \sqrt{-1} a_k)$ for some $a_k \in \QQ \cap [0,1)$. Denote:
\[
    \age(u) := \sum_{k=1}^N a_k.
\]
Then for the inverse element $u^{-1}$ we have
\[
    \age(u) + \age(u^{-1}) = N - \dim\Fix(u) = d_u.
\]
For any homogeneous $p \in \CC[\bx]$ assume the element $\lfloor p\rfloor \xi_{u} \in\A_{tot}$. Define its {\itshape left charge} $q_l$  and {\itshape right charge} $q_r$ to be 
\[
    (q_l, q_r) = \left( \frac{\deg p- d_{u}}{n}+{\rm\age}{(u)}, \frac{\deg p- d_{u}}{n}+{\rm\age}{(u^{-1})} \right).
\]
This definition endows $\A_{tot}$ with the structure of a $\QQ$-bigraded vector space. 
This is exactly the bigrading introduced in \cite{Muk, IV90}. It follows immediately that $q_\bullet (\xi_u) + q_\bullet (\xi_v) = q_\bullet (\xi_{uv})$ for $u,v\in G$, s.t. $I_u^c \cap I_v^c = \emptyset$. 

Denote by $h^{p,q}(f,G)$ the dimension of the space of bigrading $(p,q)$ elements of $\A_{f,G}$.

\section{Mirror map}
In this section we define the mirror map $\tau: \A_{tot} \to \A_{tot}$. It will be used later in Theorem~\ref{theorem: general} to set up a mirror isomorphism.
First consider the following examples. We do not give all computations explicitly, referring interested reader to Examples section of \cite{BI}.

\subsection{Examples}\label{section: mirror map examples}
Example~1 shows that there is no mirror map if $N \neq n$. 
Example~2 considers the case when there is a mirror map, interchanging broad and narrow sectors. 
Example~3 deals with the symmetry group that does not satisfy PC condition of Ebeling--Gusein--Zade. Even though, the mirror map exists.
Example~4 depicts the situation when mirror map exists but does not always interchange broad and narrow sectors as it is for diagonal symmetry groups.

Examples 2 and 3 are the particular cases of Theorem~\ref{theorem: general} and Example~4 is the particular case of Theorem~\ref{theorem: quintic}.

In all the examples beneath we consider the groups $G = S \ltimes \SL_f$ and $\widetilde G = S \ltimes \J$ with different $S \subseteq S_N$. There is a decomposition
\[
    \A_{f,S \ltimes \SL_f} = \A_{\SL, d} \oplus \A_{\SL, s}, \quad \A_{f,S \ltimes \langle J \rangle} = \A_{\langle J \rangle, d} \oplus \A_{\langle J \rangle, s}
\]
for $\A_{\SL, d}$ and $\A_{\J, d}$ being the direct sums of all $u$--th sectors of $\A_{f,S\ltimes\SL_f}$ and $\A_{f,S\ltimes\J}$, s.t. $u \in \id \cdot \SL_f$ and $u \in \id \cdot \J$ respectively.


\subsubsection{Example 1: $N = 3$ and $n=4$, $S = S_3$}
We have 
\begin{align}
    \A_{S_3 \ltimes \SL_f,s} & = \CC \left\langle \lfloor (x_2-x_1)^2 \rfloor \xi_{(1,2,3) t_1^2 t_2^2 t_3^2} , \lfloor (x_1 + x_2 + x_3)^2 \rfloor \xi_{(1,2,3)}\right\rangle, 
    \\
    \A_{S_3 \ltimes \J,s} & = \CC \left\langle \lfloor (x_1 + x_2 + x_3)^2 \rfloor \xi_{(1,2,3)} \right\rangle
\end{align}
showing that there is no mirror map between
$\A_{f,S_3 \ltimes \SL_f}$ and $\A_{f, S_3 \ltimes \J}$ when $N \neq n$.

\subsubsection{Example 2: $N = n = 5$ and $S = \langle (1,2)(3,4) \rangle \subset S_5$}
Denote
\begin{align}
    \phi_{a,b,c}(\bx) := (x_1+x_2)^a(x_3+x_4)^bx_5^c, \quad a,b,c \ge 0.
\end{align}

We have 
\begin{align}
    & \A_{\J,s}  = \bigoplus_{k=1}^4 \CC \langle \xi_{(1,2)(3,4)J^k} \rangle 
    \bigoplus_{\substack{0 \le a,b,c \le 3 \\ a+b+c = 2 \\ \text{or } a+b+c = 7} }\CC\langle \lfloor \phi_{a,b,c}\rfloor \xi_{(1,2)(3,4)} \rangle,
    \\
    & \A_{\SL,s} = \bigoplus_{\substack{0 \le a,b,c \le 3 \\ a+b=2 \\ \text{or } a+b=7}} \CC \langle \lfloor \phi_{a,a,b}\rfloor \xi_{(1,2)(3,4)}\rangle
    \bigoplus_{\substack{1 \le a,b \le 4 \\ a + b \not\in 5\ZZ}} \CC\langle \xi_{(1,2)(3,4) t_1^at_3^bt_5^{5-a-b}} \rangle.
\end{align}

The vector space isomorphism $\A_{f,S\ltimes \SL_f} \to \A_{f,S\ltimes \J}$ is
\begin{align}
    \lfloor H^k\rfloor \xi_{\id} & \mapsto \xi_{J^{k-1}},
    \\
    \xi_{(1,2)(3,4) t_1^at_4^bt_5^c} &\mapsto \lfloor\phi_{a-1,b-1,c-1}\rfloor \xi_{(1,2)(3,4)}
    \\
    \lfloor \phi_{a,a,c}\rfloor \xi_{(1,2)(3,4)} & \mapsto \xi_{(1,2)(3,4)J^{c-1}}.
\end{align}
This map interchanges bidegree $(1,2)$ classes with bidegree $(1,1)$ classes and bidegree $(2,1)$ classes with bidegree $(2,2)$ classes.

The group considered can be diagonalized, however this requires also the change of the polynomial $f$ so that we will not have the relation $\widetilde f = f$ anymore. 

\subsubsection{Example 3: $N=n=5$, $S = \langle (1,2,3), (1,2) \rangle \subset S_5$}
Denote
\begin{align}
    \phi_{a,b,c}(\bx) := (x_1+x_2+x_3)^ax_4^bx_5^c, \quad a,b,c \ge 0.
\end{align}

We have 
\begin{align}
    & \A_{\J,s}  = \bigoplus_{k=1}^4  \CC \langle \xi_{(1,2,3)J^k} \rangle 
    \bigoplus_{\substack{0\le a,b,c \le 3 \\ a+b+c = 2 \\ \text{or } a+b+c = 7} }\CC\langle \lfloor \phi_{a,b,c}\rfloor \xi_{(1,2,3)} \rangle,
    \\
    & \A_{\SL,s} = \bigoplus_{\substack{a+b=2 \\ \text{or } a+b=7}}  \CC \langle \lfloor \phi_{a,b,b}\rfloor \xi_{(1,2,3)}\rangle
    \bigoplus_{\substack{1 \le a,b \le 4 \\ a + b \not\in 5\ZZ}}\CC\langle \xi_{(1,2,3) t_1^{5-a-b}t_4^at_5^b} \rangle.
\end{align}

The vector space isomorphism $\A_{f,S\ltimes \SL_f} \to \A_{f,S\ltimes \J}$ is
\begin{align}
    \lfloor H^k\rfloor \xi_{\id} & \mapsto \xi_{J^{k-1}},
    \\
    \xi_{(1,2,3) t_1^at_4^bt_5^c} &\mapsto \lfloor\phi_{a-1,b-1,c-1}\rfloor \xi_{(1,2,3)}
    \\
    \lfloor \phi_{a,b,b}\rfloor \xi_{(1,2,3)} & \mapsto \xi_{(1,2,3)J^{b-1}}.
\end{align}
This map interchanges bidegree $(1,2)$ classes with bidegree $(1,1)$ classes and bidegree $(2,1)$ classes with bidegree $(2,2)$ classes.

\subsubsection{Example 4: $N=n=5$, $S = S_5$}
Denote 
\begin{align*}
    \phi_{a,b,c}(\bx) &:= (x_1+x_2+x_3)^a (x_4^bx_5^c + x_4^cx_5^b), 
    \\
    \psi_{a,b}(\bx) &:= (x_1+x_2+x_3)^a (x_4x_5)^b.
\end{align*}
We have $\A_{\J,s} = \A_{\J,s}^{(1)} \oplus \A_{\J,s}^{(2)}$, $\A_{\SL,s} = \A_{\SL,s}^{(1)} \oplus \A_{\SL,s}^{(2)}$ with
\begin{align}
    & \A_{\J,s}^{(1)} = \CC\langle \lfloor \phi_{0,0,2}\rfloor \xi_{(1,2,3)}, \lfloor \phi_{1,0,1}\rfloor \xi_{(1,2,3)}, \lfloor\phi_{2,2,3}\rfloor \xi_{(1,2,3)}, \lfloor\phi_{3,1,3}\rfloor \xi_{(1,2,3)} \rangle,
    \\
    & \A_{\J,s}^{(2)} = \CC \langle \lfloor \psi_{0,1} \rfloor \xi_{(1,2,3)},\lfloor \psi_{1,3}\rfloor\xi_{(1,2,3)}, \lfloor \psi_{2,0}\rfloor\xi_{(1,2,3)}, \lfloor \psi_{3,2}\rfloor\xi_{(1,2,3)} \rangle,
    \\
    & \A_{\SL,s}^{(1)} = \CC \langle \xi_{(1,2,3) t_1t_4t_5^3}, \xi_{(1,2,3) t_1^2t_4t_5^2},\xi_{(1,2,3) t_1^3t_4^3t_5^4}, \xi_{(1,2,3) t_1^4t_4^2t_5^4}\rangle
    \\
    & \A_{\SL,s}^{(2)} =  \CC \langle \lfloor \psi_{0,1} \rfloor \xi_{(1,2,3)},\lfloor \psi_{1,3}\rfloor\xi_{(1,2,3)}, \lfloor \psi_{2,0}\rfloor\xi_{(1,2,3)}, \lfloor \psi_{3,2}\rfloor\xi_{(1,2,3)} \rangle .
\end{align}
We see that $\A_{\J, d}$ contains only narrow sectors, $\A_{\J,s}$ only broad sectors. $\A_{\SL,d}$ contains only broad sectors while $\A_{\SL,s}$ is a direct sum of both broad and narrow sectors.

It's easy to guess the vector space isomorphism $\A_{f,S_5\ltimes \SL_f} \to \A_{f,S_5 \ltimes \J}$ in this case.
\begin{align}
    \lfloor H^k\rfloor \xi_{\id} & \mapsto \xi_{J^{k-1}},
    \\
    \xi_{(1,2,3) t_1^at_4^bt_5^c} &\mapsto \lfloor\phi_{a-1,b-1,c-1}\rfloor \xi_{(1,2,3)}
    \\
    \lfloor \psi_{a,b} \rfloor \xi_{(1,2,3)} & \mapsto \lfloor \psi_{a,b} \rfloor \xi_{(1,2,3)} 
\end{align}
The first two lines of it interchange degree $(1,2)$ and $(2,1)$ classes with the degree $(1,1)$ and $(2,2)$ classes respectively. However the last line sets up the isomorphism $\A_{\J,s}^{(2)} \to \A_{\SL,s}^{(2)}$ mapping the degree $(1,1)$ and $(2,2)$ classes to the degree $(1,1)$ and $(2,2)$ classes again.

\subsection{Definition of the mirror map}\label{section: mm}
For any $u \in S \ltimes G_f^d$ let $u = \prod_{a=1}^p \sigma_a g_a$ be its generalized cycle decomposition. Arbitrary element $X \in \A_{f,u}'$ reads
\[
 X = \prod_{a=1}^p  \lfloor \widetilde x_{i_a}^{r_a} \rfloor \xi_{\sigma_a g_a}.
\]
with $r_a = 0$ if $\sigma_ag_a$ is non--special and $0\le r_a \le N-2$ if $\sigma_ag_a$ is special.

Set 
\[
 \tau\left( X \right) := \prod_{a=1}^p \tau\left(\lfloor \widetilde x_{i_a}^{r_a} \rfloor \xi_{\sigma_a g_a}\right)
\]
for the product on the right hand side understood in the sense of Section~\ref{section: preliminaries}. The map $\tau$ is defined on the generalized cycles as follows.

\subsubsection{Case 1:  $u = \sigma_1 \cdot g_1$ is a non--special cycle}
Assume $g_1 = \prod_{p} t_{p}^{d_{1,p}}$ with $p$ running over $I^c_{g_1}$. Denote $d_1 := \sum_p d_{1,p} \mod N$ with $1 \le d_1 \le N$. Set
\begin{align}
    \tau(\xi_{\sigma_1 \cdot g_1}) &:= \lfloor \widetilde x_{i_1}^{d_1-1} \rfloor \cdot \xi_{\sigma_1}.
\end{align}

\subsubsection{Case 2:  $u = \sigma_1 \cdot g_1$ is a special cycle}
For any $0 \le r_1 \le N-2$ set
\begin{align}
    \tau(\lfloor \widetilde x_{i_1} ^{r_1} \rfloor \xi_{\sigma_1 \cdot g_1}) &:= \xi_{\sigma_1 \cdot h}
\end{align}
for $h := (\prod_a t_a )^{k_1}$ with $a$ running over $I^c_{\sigma_1}$ and $k_1 \in 1,\dots,N-1$, the unique integer, s.t. $r_1+1 \equiv k_1 |\sigma_1|$ modulo~$N$. 

\subsection{Mirror map of Krawitz}
The mirror map $\tau$ generalizes the mirror map of Krawitz. Namely, he defines (cf. \cite{K09}) $\tau_{Kr}: \A_{f,\SL_f} \to \A_{f,\J}$ by the following rule
\[
    \prod_{a=1}^p x_p^{r_p} \xi_\id \mapsto \xi_{\prod_{a=1}^p t_p^{r_p+1}}, \quad \xi_{\prod_{a=1}^p t_p^{d_p}} \mapsto \prod_{a=1}^p x_p^{d_p-1} \xi_\id
\]
Let's see that our mirror map coincides with the mirror maps of Krawitz on $\A'_{f,u}$, s.t. $u = \id \cdot g$. 
The generalized cycle decomposition of such elements is given by length $1$ cycles $(a)$ so that we have $u = \prod_{a=1}^N (a) t_a^{d_a}$ for $g = \prod t_a^{d_a}$. We have $x_a = \widetilde x_a$ if $d_a \equiv 0$.

When $g = \id$ we have
\begin{align}
    & \tau(\prod_{a=1}^p \lfloor x_a^{r_a} \rfloor \xi_\id) = \prod_{a=1}^p \tau( \lfloor \widetilde x_a^{r_a} \rfloor \xi_\id ) 
    = \prod_{a=1}^p \left( \lfloor 1\rfloor \xi_{t_a^{r_a+1}}\right)
    = \lfloor 1\rfloor \xi_{\prod_{a=1}^p t_a^{r_a+1}}.
\end{align}
For $d_a \neq 0$ for all $a$ we have
\begin{align}
    & \tau(\lfloor 1 \rfloor \xi_{\prod_{a=1}^p t_a^{d_a}}) = \prod_{a=1}^p \tau( \lfloor 1 \rfloor \xi_{t_a^{d_a}}) 
    = \prod_{a=1}^p \lfloor \widetilde x_a^{d_a-1} \rfloor \xi_\id.
\end{align}

\section{The vector space structure}

In this section consider prime $N=n$ and arbitrary $S \subseteq S_N$. 
The aim of this section is to describe the structure of the vector spaces $\A_{f,S\ltimes \SL_f}$ and $\A_{S\ltimes \J}$. 

According to the definition, we should consider the sets of representatives of the conjugacy classes $\C^{S\ltimes\SL_f}$ and $\C^{S\ltimes\J}$. Let $\C^S$ be some set of representatives of the conjugacy classes of $S$.
We have
\[
    \C^{S \ltimes \J} = \lbrace \sigma \cdot J^k \ | \ \sigma \in \C^S, k = 0,\dots,N-1 \rbrace.
\]

The map $S\ltimes\SL_f\to S$ is compatible with the conjugation action. We chose representatives in $\C^{S\ltimes\SL_f}$ in a way compatible with the choice of $\C^S$. The set $\C^{S\ltimes\SL_f}$ is described in Proposition~\ref{prop: SL narrow-broad} below. 

Concerning the bases of $\A_{f,S\ltimes\SL_f}$, $\A_{f,S\ltimes\J}$ we show the following proposition.

\begin{proposition}\label{proposition: structure of the phase spaces}~
    \begin{enumerate}
     \item[(A).]     The basis of $\A_{f,S \ltimes \SL_f}$ can be chosen to consists of narrow vectors $\prod_a \xi_{\sigma_a t_{i_a}^{d_a}} $ and broad vectors $\prod_a \lfloor \widetilde x_{i_a}^{r_a} \rfloor \xi_{\sigma_a}$,
     \item[(B).]     The basis of $\A_{f,S \ltimes \J}$ can be chosen to consists of narrow vectors $\prod_a \xi_{\sigma_a J^q}$ and broad vectors $\prod_a \lfloor \widetilde x_{i_a}^{r_a} \rfloor \xi_{\sigma_a}$,
    \end{enumerate}
    with $0 \le r_a \le N-2$, $1 \le d_a \le N-1$ and $1 \le q \le N-1$ in both cases.
\end{proposition}

Note that this proposition does not guarantee that all vectors of the form given do appear in the bases of $\A_{f,S\ltimes \SL_f}$ and $\A_{f,S\ltimes \J}$.
The proof is given in the following sections beneath.

\subsection{The group $G = S \ltimes \SL_f$}

\begin{proposition}\label{long cycle invariants}
    For any $g \in \SL_f$ and length $N$ cycle $\sigma$ we have $\left( \A_{f,\sigma \cdot g}' \right)^{Z(\sigma \cdot g)} = 0$.
\end{proposition}
\begin{proof}
    The element $\sigma \cdot g$ is special with $1$--dimensional fixed locus. We have ${ \A_{f,\sigma \cdot g}' \cong \CC \langle  \lfloor \widetilde x_1^p \rfloor \xi_{\sigma \cdot g} \rangle }$ for $p = 0,\dots,N-2$.
    We also have $J \in G$ and $J \in Z(\sigma \cdot g)$. The action of $J$ gives
    \[
        J^* \left( \lfloor \widetilde x_1^p \rfloor \xi_{\sigma \cdot g} \right) = \zeta_N^{p+1} \cdot \lfloor \widetilde x_1^p \rfloor \xi_{\sigma \cdot g},
    \]
    because $\det(J)=1$. The vector assumed is invariant under the action $J$ if and only if $p+1 \equiv 0 \mod N$, what never holds for the range of $p$ given.
    
\end{proof}

In the obvious way one gets that for a length $N-1$ cycle $\sigma$, an element $\sigma \cdot g$ is conjugate to $\sigma \cdot t_{i_1}^{d_1} t_{i_2}^{d_2}$ in the group $G$ for $d_1+ d_2 \equiv 0$ mod $N$ and $i_1 \in I_{\sigma}^c$, $i_2 \not\in I_\sigma^c$. This can be generalized to the following statement.

\begin{proposition}\label{prop: SL narrow-broad}
    Let $u = \sigma \cdot g \in G$ be an element with the cycle decomposition $\sigma = \prod_{a=1}^m \sigma_a$ with $m \ge 2$. Let $i_a$ be the first index of $I_{\sigma_a}^c$. Then for some $d_a \in \ZZ$, $u$ is conjugate to $\prod_{a=1}^m \sigma_a t_{i_a}^{d_a}$ in the group~$G$.
\end{proposition}
\begin{proof}
    Consider $\sigma_a = (i_1, i_2, \ldots, i_p)$ and $g_a=t_{i_1}^{d_{i_1}}t_{i_2}^{d_{i_2}}\ldots t_{i_p}^{d_{i_p}}$. For $p=1$ there is nothing to prove. For $p \ge 2$ it is sufficient to consider the bigger group $S \ltimes G^d_f$, since it is generated by its center and $G$. Now we can conjugate by $h=t_{i_2}^{d_{i_2}+d_{i_3}+\ldots+d_{i_p}}t_{i_3}^{d_{i_3}+\ldots+d_{i_p}}\ldots t^{d_{i_p}}_{i_p}$ to get the desired result.
\end{proof}
The following corollary makes a clear distinction between broad and narrow sectors for the group $S \ltimes \SL_f$.
\begin{corollary}\label{corollary: narrow classes in SL}
    Let $X \in \A_{S \ltimes \SL_f}$ be a non--zero element with $X = \lfloor \phi(\bx) \rfloor \xi_u$, then $\lfloor \phi(\bx)\rfloor$ is non--constant only if $u$ is conjugate to $\prod_{a}\sigma_a$. 

    In particular for every such $X$, the conjugacy class of $u$ contains exactly one ${v=wuw^{-1}}$, s.t. $\lfloor \phi(w\cdot\bx) \rfloor \xi_v \in \A_{S}$.
\end{corollary}
\begin{proof}
Put $u=\sigma g$ with $\sigma\in S, g\in G^d$. If $\lfloor \phi(\bx)\rfloor$ is non constant, then at least one of the cycles of $\sigma$ is special. 
Let it be $\sigma_1$.
Assume further, that there is a nonspecial cycle $\sigma_2$ in the cycle decomposition of $\sigma$.

Let $d$ be such that $|\sigma_1|+d|\sigma_2|\equiv 0\mod N$. Consider the element 
$h= (\prod_{i} t_i) (\prod_j t_j)^d \in\SL_f$ with $i$ running over $I_{\sigma_1}^c$ and $j$ running over $I_{\sigma_2}^c$.
Then $h$ commutes with $u$. We have $h^*\lfloor \tilde{x}_{i_1}^r\rfloor\xi_u=\zeta_N^{r+1}\xi_u$. Since, $h$ fixes other variables and $r$ varies between $0$ and $n-2$, there are no invariant elements in the sector. 
\end{proof}

This completes proof of Proposition~\ref{proposition: structure of the phase spaces}, case (A).

\begin{remark}
The conjugacy classes of $S_N$ are indexed by the partitions of $N$. In particular, any length $N$--cycle of $S_N$ is conjugate to $\sigma = (1,2,\dots, N)$.
We conclude that the representatives of the conjugacy classes of $S_N \ltimes \SL_f$ are given by two partitions of $n=N$ of the same length. The first one being given by the cycle type of $\sigma$ and the second by the set of exponents $d_\bullet$. 
\end{remark}

\subsection{The group $\widetilde G = S\times \langle J\rangle$}
Let $u = \sig \cdot J^q$  and $\sigma = \prod_{a=1}^p \sigma_a$ be the cycle decomposition. 
If $q\equiv 0$, all $\sigma_a$ are special. If $q \not\equiv 0$, for $n=N$ being a prime number, either $p=1$ and $\sigma = \sigma_1$ is special again, or all $\sigma_a$ are non--special.

For the group assumed $J \in Z(u)$. Consider the action of this element on $\A'_{f,u}$.

Assume $X = \lfloor \phi(\bx) \rfloor \xi_{\sigma \cdot J^0}$ and $Y = \lfloor \widetilde x_{i_1} ^{r_1}\rfloor \xi_{\sigma_1 \cdot J^q}$ with some $q \not\equiv 0$. Up to a constant multiple we have for some $0 \le r_{a} \le N-2$
\begin{align}
    X = \prod_{a = 1}^p \lfloor \widetilde x_{i_a} ^{r_a}\rfloor \xi_{\sigma_a} \quad& \text{ giving }\quad J^* \left( X \right) = \zeta_N^{\sum_a (r_a + 1)} \cdot X,
    \\
    Y = \lfloor \widetilde x_{i_1} ^{r_1}\rfloor \xi_{\sigma_1 J^q} \quad& \text{ giving }\quad J^* \left( Y \right) = \zeta_N^{r_1 + 1} \cdot Y
\end{align}
Therefore $X$ is only $J$--invariant if $\sum_a (r_a +1) \in N\ZZ$ and $Y$ is not ever $J$--invariant.

Similarly we have that $\xi_{\sigma\cdot J^q}$ is $J$--invariant for any $q \not\equiv 0$.

This gives the following proposition beneath and completes proof of Proposition~\ref{proposition: structure of the phase spaces}, case (B).

\begin{proposition}\label{prop: J narrow-broad}
    Consider $X := \lfloor \phi(\bx) \rfloor \xi_{\sigma\cdot g} \in \A'_{S\ltimes \langle J \rangle}$. Assume $X$ to be $\langle J \rangle$--invariant. Then either $\lfloor \phi(\bx) \rfloor$ is a constant or $g = \id$.
\end{proposition}

%

\subsection{Symmetric group action}
For any $u = \sigma\cdot g \in G \subseteq S_N \ltimes G_f^d$ we have $u \in Z(u)$ and for any $X = \lfloor \phi(\bx)\rfloor \xi_u$ we have $\lfloor \phi(u \cdot \bx) \rfloor = \lfloor \phi(\bx) \rfloor$. However the action of $u$ on the generator $\xi_u$ is not necessarily trivial. It was proved in Corollary~39 of \cite{BI}, or can be deduced from Eq.~\eqref{eq: group action}, that we have
\[
    u^* \left( \xi_u\right) = (-1)^{\sgn(\sigma)} \xi_u.
\]
Consider the $\langle u \rangle$--invariant subspace of $\A'_{f,u}$.
\begin{equation}
    \left( \A'_{f,u}\right)^{\langle u \rangle} = 0 \ \text{ if } \ \sigma \text{ is odd and } \ \left( \A'_{f,u}\right)^{\langle u \rangle} = \A'_{f,u} \text{ otherwise}.
\end{equation}

Due to this reason no sector of $u = \sigma \cdot g$ with $\sigma$ odd appears in $\A_{f,G}$. However such elements $u$, if they exist in the group, can still affect the conjugacy classes decomposition and also the centralizers of the other group elements, contributing non--trivially to $\A_{f,G}$. 

In the following proposition we consider group--action approach to the parity condition of Ebeling--Gusein-Zade.
\begin{proposition}\label{prop: PC action}
    Let $u,v \in A_N$ be even commuting permutations. For $T := \langle u,v\rangle$ we have. 
    \begin{itemize}
     \item PC holds for $T$ if and only if $u^*(\xi_v) = \xi_v$ and $v^*(\xi_u) = \xi_u$,
     \item if PC doesn't hold for $T$ we have $u^*(\xi_v) = -\xi_v$ or $v^*(\xi_u) = -\xi_u$.
    \end{itemize}
\end{proposition}
\begin{proof}
 Let $u$ and $v$ share a common eigenbasis $\widetilde x_1, \dots, \widetilde x_N$. Let $u(\widetilde x_k) = \lambda'_k \widetilde x_k$ and ${v(\widetilde x_k) = \lambda''_k \widetilde x_k}$. Assume also $\lambda'_1 = \dots = \lambda_p' = 1$ and $\lambda''_1 = \dots = \lambda_q'' = 1$. We have
\begin{align}
    v^*\left( \xi_u \right) &= \frac{1}{\lambda'_{p+1}\cdots \lambda'_N} \ \xi_u = \frac{1}{\lambda'_{p+1}\cdots \lambda'_N} \ \xi_u \quad q \le p
    \\
    v^*\left( \xi_u \right) &= \frac{1}{\lambda'_{q+1}\cdots \lambda'_N} \ \xi_u = \frac{\lambda'_{p+1}\cdots \lambda'_q}{\lambda'_{p+1}\cdots \lambda'_N} \ \xi_u \quad q > p.
\end{align}
It follows that for both cases
\begin{equation}
    v^*\left( \xi_u \right) = \frac{\lambda_1' \cdots \lambda_q'}{\det(v)} \xi_u = \frac{\det(v_{\mid \Fix(u)})}{\det(v)} \xi_u,
\end{equation}
where $v_{\mid \Fix(u)}$ stands for the restriction of $v$ to the fixed locus of $u$. Assume $q > p$. The action of $v$ on $\Fix(u)$ is given by a permutation of the set $\{1,\dots,q\}$. Therefore both determinants assumed are either $+1$ or $-1$. We have
\[
    \det(v) = (-1)^{N - \dim\Fix(v)}, \quad \det(v_{\mid \Fix(u)}) = (-1)^{\dim\Fix(u) - \dim\Fix(v)}.
\]
This completes the proof.
\end{proof}

\section{Mirror isomorphism}\label{section: general case proof}

We have seen in Example 1 of Section~\ref{section: mirror map examples} that the condition $N=n$ is necessary for mirror symmetry to hold. Assume it to hold true. Assume further $N$ to be prime.
Fix some $S \subseteq S_N$. 

Let $G \subset S \ltimes G_f$. We call $u$--th sector of $\A_{f,G}$ \textit{stable} if 
\begin{equation}
    \sigma^*(\xi_u) = \xi_u \quad \text{ for any } \quad \sigma \in Z(u) \cap S \subset G.
\end{equation}
Denote by $\A_{f,G}^{stable} \subseteq \A_{f,G}$ the direct sum of all stable sectors $\A_{f,u}$. It's important to note that our stability property considers the generator of the $u$--th sector rather than an arbitrary element $\lfloor \phi(\bx)\rfloor \xi_u$ of it.

In particular according to Proposition~\ref{prop: PC action}, we have $\A_{f,G}^{stable} = \A_{f,G}$ if $S$ satisfies \eqref{eq: PC} of Ebeling--Gusein-Zade. The converse is not true. In particular these vector spaces coincide in Examples 2 and 3 of Section~\ref{section: mirror map examples}. In the case of Example 4 we have $\A_{f,G}^{stable} = \A_{G,d} \oplus \A_{G,s}^{(1)} \subsetneq \A_{f,G}$.

\begin{theorem}\label{theorem: general}
    The map $\tau$ establishes an isomorphism $\A_{f,S \ltimes \SL_f}^{stable} \to \A_{f,S \ltimes \J}^{stable}$. 
    Under this isomorphism we have $q_l(X) = q_l(\tau(X))$ and $q_r(X) = N-2-q_r(\tau(X))$ for any homogeneous $X \in \A_{f,S \ltimes \SL_f}^{stable}$.
\end{theorem}
\begin{proof}
 See Sections~\ref{section: narrow iso} and \ref{section: broad iso} for the first claim and Section~\ref{section: mm bidegree} for the second claim.
\end{proof}

To show that mirror map sets up an isomorphism we need to consider the action of the centralizers of the elements $u = \sigma g \in S \ltimes \SL_f$ and $v = \sigma J^k \in S \ltimes \J$ with the same $\sigma \in S$.

\subsection{Narrow sectors}\label{section: narrow iso}
Let $u \in S \ltimes \SL_f$ be s.t. $u$--th sector is stable and narrow. Assume $u$ to be decomposed into generalized cycles as $u = \prod_{a=1}^p \sigma_a t_{i_a}^{d_a}$. Let $Y := \xi_u \in \A'_{f,S \ltimes \SL_f}$. Then $X := \tau(Y) = \prod_{a} \lfloor \widetilde x_{i_a}^{d_a-1} \rfloor \xi_{\sigma_a} \in \A'_{f,S \ltimes \J}$. 

As $u$-th sector is narrow, any $v \in Z(u) \subset S \ltimes \SL_f$ gives $v^* (Y) = (\det(v))^{-1} Y = Y$ and we know that $Y$ is non--zero in $\A_{f,S \ltimes \SL_f}$. We show that $X$ is non--zero in $\A_{f,S \ltimes \J}$ too.

The centralizer $Z(\prod_a \sigma_a) \subset S \ltimes \J$ is generated by the element $J$ and all permutations $\sigma' \in S$, commuting with $\prod_a \sigma_a$. We have that $J^*(X) = X$ if and only if $\sum_a d_a \equiv 0$ modulo $N$, what is equivalent to the condition $\prod_a t_{i_a}^{d_a} \in \SL_f$.
The action of $\sigma'$ is trivial on $X$ because the $u$--th sector is stable.

\subsection{Broad sectors}\label{section: broad iso}
Let $u \in S \ltimes \SL_f$ be s.t. $u$--th sector is stable and broad. Assume $u$ to be decomposed into generalized cycles as $u = \prod_{a=1}^p \sigma_a$. Take $X := \prod_{a} \lfloor \widetilde x_{i_a}^{d_a-1} \rfloor \xi_{\sigma_a} \in \A'_{f,S \ltimes \SL_f}$.

\begin{proposition}\label{prop: equal exponents in SL}
    Let $X = \prod_{a=1}^p \lfloor \widetilde x_{i_a}^{r_a} \rfloor \xi_{\sigma_a}$ be non--zero in $\A_{f,S \ltimes \SL_f}$. Let $k_1,\dots,k_p$ be as in Section~\ref{section: mm}. Then $k_a = k_b$ for all $1 \le a,b \le p$.
\end{proposition}
\begin{proof}
    Take $a=1$ and $b=2$. Let $g_1 := ( \prod_{a \in I_{\sigma_1}^c} t_a )^\alpha$ and $g_2 := ( \prod_{a \in I_{\sigma_2}^c} t_a)^\beta$. Then for $\alpha |\sigma_1| + \beta |\sigma_2| \equiv 0$ modulo $N$ we have $v := g_1g_2 \in Z(\prod_a \sigma_a) \subset S \ltimes \SL_f$. 
    
    We have $v^* (X) = \lambda_1^{r_1+1}\lambda_2^{r_2 +1} \cdot X$ with $\lambda_1 = \zeta_N^{\alpha}$, $\lambda_2 = \zeta_N^{\beta}$. We conclude that $X$ is invariant w.r.t. $v$ if and only if $(r_1+1) \alpha + (r_2+1)\beta \equiv 0$. For $N$ being prime this holds for $\alpha$ and $\beta$ assumed if and only if $k_1 \equiv k_2$. 
\end{proof}

It follows by the proposition above that $Y := \tau(X) = \xi_{\sigma_a J^k} \in \A'_{f,S \ltimes \J}$ is well--defined. We have $J^*(Y) = \det(J)^{-1} Y = Y$ and it remains to consider the action of symmetric part elements on both $X$ and $Y$. Every $S$--element $\sigma' \in Z(\prod_a \sigma_a J^k) \subset S \ltimes \J$ gives $(\sigma')^* Y = \det(\sigma')^{-1} Y = Y$. At the same time we have $\sigma' \in Z(u)$ and $(\sigma')^* X = X$ because the sector is stable.

\subsection{Bidegree}\label{section: mm bidegree}

Let $X = \xi_{u} \in \A_{tot}$ with narrow $u = \prod_{a=1}^p \sigma_a g_a$. Set $Y := \tau(X) \in \A_{tot}$. Let $d_1,\dots,d_p$ be the exponents associated to $g_1,\dots,g_p$ as in Section~\ref{section: mm}. We have
\begin{align}
    q_l(Y) &= q_r(Y) = \sum_{a=1}^p \frac{d_a-1}{N} + \frac{N-p}{2} - \frac{N-p}{N} = \sum_{a=1}^p \frac{d_a}{N} - 1 + \frac{N-p}{2},
    \\
    q_l(X) &= \frac{N-p}{2} + \sum_{a=1}^p\frac{d_a}{N} -1, \quad q_r(X) = \frac{N+p}{2} - \sum_{a=1}^p \frac{d_a}{N} - 1.
\end{align}
This gives the desired statement for the map $\tau$, sending narrow sector $X \in \A'_{f,S\ltimes\SL_f}$ to $Y = \tau(X) \in \A'_{f,S\ltimes\J}$.

Because $\tau$ is involutive, this also approves the case when $\tau$ maps a broad sector element of $\A'_{f,S\ltimes\SL_f}$ to a narrow element of $\A'_{f,S\ltimes\J}$. However we find it useful to provide the full prove in this case too.

Let $X = \prod_{a=1}^p \lfloor x_{i_a}^{r_a} \rfloor \xi_{\sigma_a}$ and $Y := \tau(X) = \prod_{a=1}^p \xi_{\sigma_a J^k}$.
According to the defintion of the mirror map we have for some $l_a \ge 0$ that $r_a+1 = k \cdot |\sigma_a| - l_a N$. Moreover we can assume $l_a < |\sigma_a|$ due to our assumptions on $k$ and $r_a$.

We have 
\[
    q_l(X) = q_r(X) = \frac{\sum_{a=1}^p r_a}{N} + \frac{N-p}{2} - \frac{N-p}{N} = \frac{\sum_{a=1}^p (r_a + 1)}{N} -1 + \frac{N-p}{2}.
\]
The bigrading of $Y$ has more complicated expression. The eigenvalues of $\prod_{a=1}^p\sigma_a J^k$ are
\[
    \lambda_{a,q} := \exp (2\pi \sqrt{-1} \left( \frac{q}{|\sigma_a|} + \frac{k}{N} \right) )
\]
for $a=1,\dots,p$ and $q=1,\dots,|\sigma_a|$. To compute $\age(\prod_a \sigma_a J^k)$ note that that the expression in the brackets is not always smaller than $1$. We have
\[
    \lambda_{a,q} := \exp ( \frac{2\pi \sqrt{-1}}{|\sigma_a|} \left( q + l_a + \frac{r_a+1}{N} \right) ).
\]
Because $r_a+1 \le N-1$, we conclude that there are exacly $l_a$ values of $q$, s.t. the expression inside brackets is greater of equal to $|\sigma_a|$. This gives
\begin{align}
    \age(\prod_a \sigma_aJ^k) &= \sum_{a=1}^p \left( \sum_{q=1}^{|\sigma_a|} \frac{q}{|\sigma_a|} + |\sigma_a| \frac{k}{N} - l_a \right)
    = \sum_{a=1}^p \left( \frac{|\sigma_a|-1}{2} + \frac{k |\sigma_a|}{N} - l_a \right).
\end{align}

Summing up we have $q_l(Y) = \frac{N-p}{2} + k - \sum_{a=1}^p l_a - 1 = q_l(X)$. It's easy to see that $N-2 = q_l(Y) + q_r(Y)$, what finishes the proof.


%
\section{Miror map for Fermat quintic}
The purpose of this section is to prove the following theorem.
\begin{theorem}\label{theorem: quintic}
    For $N=5$ there is a vector space isomorphism $\A_{f,S \ltimes \SL_f} \to \A_{f,S \ltimes \J}$.
\end{theorem}
\begin{corollary}
    Denote $h^{p,q}_\SL:= h^{p,q}(f,S\ltimes\SL)$ and $h^{p,q}_\J := h^{p,q}(f,S\ltimes\J)$. 
    We have
    \[
        h^{1,1}_\SL + h^{2,1}_\SL = h^{1,1}_\J + h^{2,1}_\J.
    \]
\end{corollary}
\begin{proof}
    In both spaces the subspaces of $(0,0)$, $(3,3)$, $(3,0)$ and $(0,3)$ classes are all $1$--dimensional.
    Both $\A_{f,S\ltimes\SL}$ and $\A_{f,S\ltimes\J}$ have a pairing, respecting the bidegree (cf. Theorem~32 of \cite{BI}).
    The statement follows now from the equality of the dimensions of $\A_{f,S\ltimes\SL}$ and $\A_{f,S\ltimes\J}$.
\end{proof}

The proof of the theorem above occupies the rest of this section. 

To show the theorem we first need to redefine the mirror map.

\subsection{Definition}
As in general case set $\widehat\tau\left( X \right) := \prod_{a=1}^p \widehat\tau\left(\lfloor \widetilde x_{i_a}^{r_a} \rfloor \xi_{\sigma_a g_a}\right)$ for $X = \prod_{a=1}^p  \lfloor \widetilde x_{i_a}^{r_a} \rfloor \xi_{\sigma_a g_a}$. 
We (re)define $\widehat\tau$ on the generalized cycles. 

\subsubsection{Case 1:  $u = \sigma_1 \cdot g_1$ is a non--special cycle}
Assume also $g = \prod_{p} t_{p}^{d_{1,p}}$ with $p$ running over $I^c_{g}$. Denote $d_1 := \sum_p d_{1,p} \mod N$ with $1 \le d_1 \le N-1$.
\begin{align}
    \widehat\tau(\xi_{\sigma \cdot g}) &:= \widehat\tau_1(\xi_{\sigma \cdot g}) + \widehat\tau_2(\xi_{\sigma \cdot g}),
    \\
    \text{for} \quad & \widehat\tau_1(\xi_{\sigma \cdot g}) := \lfloor \widetilde x_{1}^{d_1-1} \rfloor \cdot \xi_{\sigma}, \quad \widehat\tau_2(\xi_{\sigma \cdot g}) := \xi_{\sigma \cdot h},
\end{align}
for $h = (\prod_a t_a )^{k_1}$ and $k_1 \in 1,\dots,N-1$, the unique integer, s.t. $d_1 \equiv k_1 |\sigma|$ modulo~$N$. 

\subsubsection{Case 2:  $u = \sigma \cdot g$ is a special cycle}
We have necessarily $g = \id$ or $\sigma = \id$. For any $0 \le r_1 \le N-2$ set
\begin{align}
    \widehat\tau(\lfloor \widetilde x_{i_1} ^{r_1} \rfloor \xi_{\sigma \cdot g}) &:= - \widehat\tau_3(\lfloor \widetilde x_{i_1} ^{r_1} \rfloor \xi_{\sigma \cdot g}) + \widehat\tau_4(\lfloor \widetilde  x_{i_1} ^{r_1} \rfloor \xi_{\sigma \cdot g}),
    \\
    \text{for} \quad & \widehat\tau_3(\lfloor \widetilde  x_{i_1} ^{r_1} \rfloor \xi_{\sigma \cdot g}) := \lfloor \widetilde  x_{i_1} ^{r_1} \rfloor \xi_{\sigma \cdot g}, \quad \widehat\tau_4(\lfloor \widetilde x_{i_1} ^{r_1} \rfloor \xi_{\sigma \cdot g}) := \xi_{\sigma \cdot h},
\end{align}
with $h := (\prod_a t_a )^{k_1}$ and $k_1 \in 1,\dots,N-1$, the unique integer, s.t. $r_1+1 \equiv k_1 |\sigma|$ modulo~$N$. 

\subsection{The properties of $\widehat\tau$}
The map $\widehat\tau$ mixes up the special and non--special generators $\xi_{\sigma_ag_a}$. However we have in $\A_{f,S\ltimes \SL_f}$ and $\A_{f,S\ltimes\J}$ only broad or narrow elements. It's easy to check that the map $\widehat\tau$ still generalizes the mirror map of Krawitz.

\begin{proposition}
    For $G$ being $\SL_f$ or $\J$, let $X,Y \in \A'_{S\ltimes G}$ be narrow and broad basis elements respectively. We have in $\A'_{f, S \ltimes G}$
    \begin{equation}
        \widehat\tau(X) = \widehat\tau_1(X) + \widehat\tau_2(X), \quad \widehat\tau(Y) = -\widehat\tau_3(Y) + \widehat\tau_4(Y).
    \end{equation}
\end{proposition}
\begin{proof}
    This follows immediately from Proposition~\ref{proposition: structure of the phase spaces}.
\end{proof}

The following proposition considers the elements for which both $\widehat\tau_1 \neq 0$, $\widehat\tau_2\neq 0$ or both $\widehat\tau_3 \neq 0$,$\widehat\tau_4\neq 0$.
\begin{proposition}\label{prop: pair of taus}
    Let $Y := \xi_v$ be a non--zero narrow element of $\A_{f,S \ltimes \SL_f}$. We have
    \begin{description}
     \item[(A)] both $\widehat\tau_1(Y) \neq 0$ and $\widehat\tau_2(Y) \neq 0$ in $\A_{f,S \ltimes \langle J \rangle}$ if and only if $\exists X$ narrow non-zero in $\A_{f,S \ltimes \SL_f}$, s.t. 
     $\widehat\tau_1(Y) = \widehat\tau_3(X)$ and $\widehat\tau_2(Y) = \widehat\tau_4(X)$.
     \item[(B)] define the following $\CC$--linear map $K: \A_{f,S \ltimes \SL_f} \to \A_{f,S \ltimes \SL_f}$. 
     Let it act by
     \[
      K(X) := Y \text{ and } K(Y) := - X
     \]
     on $X$ and $Y$ as in (A) above and act by identity otherwise.
    Then the map 
    \[
        \widetilde \tau := \frac{1}{2}\left( \widehat\tau \circ K + \widehat\tau \right)
    \]
     maps the basis elements of $\A_{f,S\ltimes \SL_f}$ to the basis elements of $\A_{f,S\ltimes \langle J \rangle}$.
    \end{description}
\end{proposition}
\begin{proof}~\\
{\bf (A).} Assume $\widehat\tau_1(Y) \neq 0$, $\widehat\tau_2(Y) \neq 0$ and $v = \sigma \cdot g$. Let $\sigma = \prod_{a=1}^p \sigma_a$ be the cycle decomposition. 

Set $X := \widehat\tau_1(Y)$. By the construction $X \in \A'_{f,u}$ with $u=\sigma$. For $d_a$ as in Section~\ref{section: mm} we have $X = \prod_{a=1}^p \lfloor x_{i_a}^{d_a-1} \rfloor \xi_u$. Let $k_a |\sigma_a| = d_a \mod N$. Because $\widehat\tau_2(Y) \neq 0$, we have that $k_1=\dots=k_p := \kappa$.

To show that $X$ is non--zero in $\A_{f,S \ltimes \SL_f}$, consider the centralizer $Z(u)$ in $S\ltimes \langle J\rangle$. It's generated by the permutations $\sigma' \in S$, commuting with $\sigma$ and also the diagonal element $J$. 
In $S \ltimes \SL_f$ the centralizer $Z(v)$ is generated by the same set of $\sigma' \in S$ as above and also the diagonal elements $h \in \SL_f$ commuting with $\sigma$. Denote $J_a := \prod_i t_i$ with $i \in I^c_{\sigma_a}$. Then for some $l_1,\dots,l_p$ we have $h = \prod_{a=1}^p J_a^{l_a}$. By Eq.~\eqref{eq: group action} we have
\[
    h^*\left( X \right) = \prod_{a=1}^p \zeta_N^{l_a d_a} \cdot X.
\]
Using the $\kappa$ above, modulo $N$ we have $\sum_{a=1}^p l_ad_a \equiv \kappa \left( \sum_{a=1}^p l_a |\sigma_a| \right) \equiv 0$ because $h \in \SL_f$.

The action of $\sigma'$ considered is the same on $X$ both considered in $\A'_{f,S\ltimes \SL_f}$ and $\A'_{f,S\ltimes \langle J \rangle}$. It is trivial because non--special $\widehat\tau_2(Y)$ also has $\sigma'$ in it's centralizer and is non--zero.

For $\widehat\tau_3(X) \neq 0$, $\widehat\tau_4(X) \neq 0$. Set $Y := \widehat\tau_4(X)$. This case is treated completely similarly.
\\
{\bf (B).} The statement is obvious for $X$ and $Y$, s.t. $\widehat\tau_1(X)$, $\widehat\tau_2(X)$ and $\widehat\tau_3(Y)$, $\widehat\tau_4(Y)$ are not simultaneously non--zero. If this doesn't hold, by case (A) above we have $\widetilde\tau(Y) = \widehat\tau_1(Y)$ and $\widetilde\tau(X) = \widehat\tau_4(X)$.
\end{proof}

\begin{proposition}
    The map $\widetilde \tau$ is an isomorphism for $N=5$.
\end{proposition}
\begin{proof}
    Let's follow the steps of Section~\ref{section: general case proof}. Let $X,Y$ be non--zero broad and narrow elements of $\A_{f,S\ltimes \SL_f}$ respectively. By using Proposition~\ref{prop: pair of taus} (B) it's enough to show that for non--zero $X,Y$ we have $\widehat\tau_1(Y)$, $\widehat\tau_2(Y)$ are not simultaneously zero and $\widehat\tau_3(X)$, $\widehat\tau_4(X)$ are not simultaneously zero.
    Moreover it's enough to consider the action of the symmetric part elements $\sigma' \in S$ on both sides.
    
    Assume $\widehat\tau_1(Y) = 0$ in $\A_{f,S\ltimes\J}$. Then there is $\sigma' \in Z(Y)$, s.t. $\sigma' \in Z(\widehat\tau_1(Y))$ and $(\sigma')^* (\widehat\tau_1(Y)) = - \widehat\tau_1(Y)$. In this case we should have $\sigma = (i,j)(k,l)$ for some pairwise distinct $1 \le i,j,k,l \le 5$. Without loss of generality assume $\{i,j,k,l\} = \{1,2,3,4\}$. Then we should have $u = \sigma g = (1,2)(3,4)t_1^{d}t_3^dt_5^{-2d}$. Then we have $\widehat\tau_2(Y) = \xi_{\sigma J^{-2d}}$ that is non--zero in $\A_{f,S\ltimes\J}$.
    The case of $\widehat\tau_4(X) = 0$ in $\A_{f,S\ltimes\J}$ is treated completely similarly. 
\end{proof}
\begin{remark}
    It's natural to consider the map $\widetilde\tau$ in a more general context. However it fails to establish an isomorphism already for $N=n=7$.
\end{remark}

\subsection{Examples}\label{section: examples2}
We follow the notation of Section~\ref{section: mirror map examples}. In particular, we make use of the spaces $\A_{\SL,s}$ and $\A_{\J,s}$.
\subsubsection{Klein 4--group}
Consider $S = \langle \sigma_1,\sigma_2 \rangle \subset S_5$ with $\sigma_1 := (1,2)(3,4)$, $\sigma_2 := (1,3)(2,4)$. Define also $\sigma_3:= \sigma_1\sigma_2$. The spaces $\A_{\SL,s}$ and $\A_{\J,s}$ are $24$--dimensional.
Let
\begin{align}
    \phi_{a,b,c}^{(1)} := \left((x_1+x_2)^a(x_3+x_4)^b-(x_1+x_2)^b(x_3+x_4)^a\right)x_5^c,
    \\
    \phi_{a,b,c}^{(2)} := \left((x_1+x_3)^a(x_2+x_4)^b-(x_1+x_3)^b(x_2+x_4)^a\right)x_5^c,
    \\
    \phi_{a,b,c}^{(3)} := \left((x_1+x_4)^a(x_2+x_3)^b-(x_1+x_4)^b(x_2+x_3)^a\right)x_5^c.
\end{align}
We have
\begin{align}
    & \A_{\SL,s} = \bigoplus_{\substack{a,b=1,\dots,4 \\ a \le b}} \CC \langle \xi_{\sigma_1 t_1^at_3^bt_5^{5-a-b}} \rangle 
    \bigoplus_{\substack{a,b=1,\dots,4 \\ a \le b}} \CC \langle \xi_{\sigma_2 t_1^at_2^bt_5^{5-a-b}} \rangle
    \bigoplus_{\substack{a,b=1,\dots,4 \\ a \le b}} \CC \langle \xi_{\sigma_3 t_1^at_4^bt_5^{5-a-b}} \rangle.
    \\
    & \A_{\J,s} = \bigoplus_{k=1}^3 \bigoplus_{a=1}^4 \CC \langle \xi_{\sigma_k J^a} \rangle \bigoplus_{k=1}^3 \CC\langle \phi_{1,0,1}^{(k)}\xi_{\sigma_k}, \phi_{2,0,0}^{(k)}\xi_{\sigma_k}, \phi_{3,1,3}^{(k)}\xi_{\sigma_k}, \phi_{3,2,2}^{(k)}\xi_{\sigma_k} \rangle.
\end{align}
The point that the polynomials $\phi_{a,b,c}^{(k)}$ are skew-symmetric w.r.t. action of $\sigma_l$ with $l\neq k$ reflects the fact that $\sigma_l^*(\xi_{\sigma_k}) = -\xi_{\sigma_k}$.

Because $\A_{\SL,s}$ only contains broad sectors, we should only consider the maps $\widetilde \tau_1$ and $\widetilde \tau_2$.
One computes
\begin{align}
    \widehat\tau_1(\xi_{\sigma_1 t_1^at_3^bt_5^{c}}) = 
    \begin{cases}
        \phi_{a-1,b-1,c-1}^{(1)} & \text{ if } a \neq b,
        \\
        0 & \text{ if } a = b.
    \end{cases}
    \quad 
    \widehat\tau_2(\xi_{\sigma_1 t_1^at_3^bt_5^{c}}) = 
    \begin{cases}
        0 & \text{ if } a \neq b,
        \\
        \xi_{\sigma_1 J^{c+1}} & \text{ if } a = b.
    \end{cases}
\end{align}
The last equality looks to be misterious because it only depends on $c$. However, the mistery is resolved by the fact that we have coincidence of the indices $a,b$ in this case and $a,b,c$ are all connected by the degree condition.

In this example the map $K$ appeared to be identity and $\widetilde \tau = \widehat \tau$.

\subsubsection{$S \cong S_3 \times S_2$}
Consider $S = \langle (1,2,3),(1,2),(4,5) \rangle \subset S_5$. The spaces $\A_{\SL,s}$ and $\A_{\J,s}$ are $8$--dimensional.
Let
\begin{align}
    \phi_{a,b,c} := (x_1+x_2+x_3)^a(x_4^bx_5^c + x_4^cx_5^b).
\end{align}
We have
\begin{align}
    & \A_{\SL,s} = \CC\langle \xi_{(1,2,3)t_1 t_4t_5^3}, \xi_{(1,2,3)t_1^2t_4t_5^2}, \xi_{(1,2,3)t_1^3t_4^3t_5^4}, \xi_{(1,2,3)t_1^4t_4^2t_5^4} \rangle
    \\
    & \quad \bigoplus \CC\langle \lfloor \phi_{0,1,1}\rfloor \xi_{(1,2,3)}, \lfloor \phi_{2,0,0}\rfloor \xi_{(1,2,3)}, \lfloor \phi_{3,2,2}\rfloor \xi_{(1,2,3)}, \lfloor \phi_{1,3,3}\rfloor \xi_{(1,2,3)} \rangle.
    \\
    & \A_{\J,s} = \bigoplus \CC\langle \lfloor \phi_{0,0,2}\rfloor \xi_{(1,2,3)}, \lfloor \phi_{1,0,1}\rfloor \xi_{(1,2,3)}, \lfloor \phi_{2,2,3}\rfloor \xi_{(1,2,3)}, \lfloor \phi_{3,1,3}\rfloor \xi_{(1,2,3)} \rangle
    \\
    & \quad \bigoplus \CC\langle \lfloor \phi_{0,1,1}\rfloor \xi_{(1,2,3)}, \lfloor \phi_{2,0,0}\rfloor \xi_{(1,2,3)}, \lfloor \phi_{3,2,2}\rfloor \xi_{(1,2,3)}, \lfloor \phi_{1,3,3}\rfloor \xi_{(1,2,3)} \rangle.
\end{align}
The mirror map gives
\begin{align}
    &\widehat\tau_3(\lfloor \phi_{a,b,c}\rfloor \xi_{(1,2,3)}) = \lfloor \phi_{a,b,c}\rfloor \xi_{(1,2,3)}, \quad \widehat\tau_4(\lfloor \phi_{a,b,c}\rfloor \xi_{(1,2,3)}) = 0, \quad b\neq c,
    \\
    &\widehat\tau_1(\xi_{(1,2,3)t_1^a t_4^bt_5^c}) = \lfloor \phi_{a-1,b-1,c-1}\rfloor \xi_{(1,2,3)}, \quad \widehat\tau_2(\xi_{(1,2,3)t_1^a t_4^bt_5^c}) = 0.
\end{align}

In this example the map $K$ appeared to be identity and $\widetilde \tau = \widehat \tau$.

\subsubsection{$S \cong S_3$ revisited}
To illustrate how the map $K$ works, consider $S := \langle (1,2,3), (1,2) \rangle$ as in Example~3 of Section~\ref{section: mirror map examples}.
We have
\begin{align}
    & \widehat\tau (\lfloor H^k\rfloor \xi_{\id})  = \xi_{J^{k-1}},
    \\
    & \widehat\tau (\xi_{(1,2,3) t_1^at_4^bt_5^c}) = \lfloor\phi_{a-1,b-1,c-1}\rfloor \xi_{(1,2,3)} ,\quad \text{ if } b\neq c,
    \\
    & \widehat\tau (\xi_{(1,2,3) t_1^at_4^bt_5^b}) = \lfloor\phi_{a-1,b-1,b-1}\rfloor \xi_{(1,2,3)} + \xi_{(1,2,3) J^b},
    \\
    & \widehat\tau(\lfloor \phi_{a,b,b}\rfloor \xi_{(1,2,3)}) = -\lfloor \phi_{a,b,b}\rfloor \xi_{(1,2,3)} +  \xi_{(1,2,3)J^{b-1}}.
\end{align}
Then $K(\xi_{(1,2,3) t_1^at_4^bt_5^b}) = -\lfloor\phi_{a-1,b-1,b-1}\rfloor \xi_{(1,2,3)}$, $K(\lfloor \phi_{a,b,b}\rfloor \xi_{(1,2,3)}) = \xi_{(1,2,3) t_1^at_4^bt_5^b}$ and $K = \id$ on all the other basis vectors.

\end{document}